\documentclass{amsproc}
\usepackage[margin=1.2in,nomarginpar]{geometry}
\usepackage{amssymb}
\usepackage{amsmath}
\usepackage{mathdots}
\usepackage{amsbsy}
\usepackage{amscd}
\usepackage{amsthm}

\usepackage{url}
\usepackage{xcolor}
\usepackage{amsmath,amssymb,amsthm,amsfonts,longtable}
\usepackage[english]{babel}
\usepackage{tikz-cd}
\usetikzlibrary{cd}
\usetikzlibrary{decorations.markings}
\tikzset{negated/.style={
		decoration={markings,
			mark= at position 0.5 with {
				\node[transform shape] (tempnode) {$\times$};
			}
		},
		postaction={decorate}
	}
}
\usepackage{array}
\usepackage[colorlinks,citecolor=red,urlcolor=blue,bookmarks=false,hypertexnames=true]{hyperref} 
\usepackage{multirow}
\usepackage{pbox}
\newtheorem{theorem}{Theorem}
\newtheorem{corollary}[theorem]{Corollary}
\newtheorem{lemma}[theorem]{Lemma}

\newtheorem{definition}[theorem]{Definition}
\newcommand{\Irr}{\textnormal{Irr}}

\newcommand{\cl}{\textnormal{cl}}
\newcommand{\nil}{\textnormal{c}}
\newcommand{\cd}{\textnormal{cd}}
\newcommand{\nl}{\textnormal{nl}}
\newcommand{\lin}{\textnormal{lin}}

\title[]{Classification of GVZ and Nested GVZ $p$-groups up to Order $p^6$}

\author{Ram Karan Choudhary$^*$}
\address{Indian Institute of Science Education and Research Pune, Dr.~Homi Bhabha Road, Pashan, Pune--411008, India}
\email{ramkchoudhary1997@gmail.com, ram.choudhary@iiserpune.ac.in}

\thanks{$^{\textbf{*}}$ Corresponding author.
}
\subjclass[2020]{primary 20D15; secondary 20D25, 20C15}
\keywords{GVZ-groups, nested groups, $p$-groups}
\begin{document}
	
	\begin{abstract}
	Let $G$ be a finite group and let $\Irr(G)$ denote the set of irreducible complex characters of $G$. For a normal subgroup $N \trianglelefteq G$ and $\chi \in \Irr(G)$, we say that $\chi$ is \emph{fully ramified} over $N$ if $\chi(g)=0$ for all $g \in G \setminus N$. A group $G$ is said to be of \emph{central type} if there exists $\chi \in \Irr(G)$ that is fully ramified over $Z(G)$. Motivated by this notion, an irreducible character $\chi \in \Irr(G)$ is called of \emph{central type} if $\chi$ vanishes on $G \setminus Z(\chi)$, where
	\[
	Z(\chi)=\{\, g \in G : |\chi(g)|=\chi(1) \,\}
	\]
	is the center of $\chi$. Groups in which every irreducible character is of central type are called \emph{GVZ-groups}. Furthermore, a group $G$ is said to be \emph{nested} if for all $\chi,\psi \in \Irr(G)$, either $Z(\chi)\subseteq Z(\psi)$ or $Z(\psi)\subseteq Z(\chi)$.
	
	It is known that a GVZ-group is nilpotent. In this article, we classify all GVZ and nested GVZ $p$-groups of order at most $p^6$, where $p$ is an odd prime.
	\end{abstract}
	\maketitle

	\section{Introduction} 
	For a finite group $G$, we write $\Irr(G)$ for the set of irreducible complex characters of $G$. Let $N \trianglelefteq G$ and $\chi \in \Irr(G)$. We say that $\chi$ is \emph{fully ramified} over $N$ if $\chi(g)=0$ for all $g \in G \setminus N$. A group $G$ is called of \emph{central type} if there exists $\chi \in \Irr(G)$ that is fully ramified over $Z(G)$. This class of groups was introduced by DeMeyer and Janusz~\cite{DJ}, who proved that $G$ is of central type if and only if each Sylow $p$-subgroup $H_p$ of $G$ is of central type and $Z(H_p)=Z(G)\cap H_p$. It is also known that groups of central type are solvable (see~\cite{HI}). These groups have been studied extensively in~\cite{DJ, Espu, Gagola, HI}.
	
	Motivated by this notion, a character $\chi \in \Irr(G)$ is said to be of \emph{central type} if it vanishes on $G \setminus Z(\chi)$, where
	\[
	Z(\chi)=\{g \in G : |\chi(g)|=\chi(1)\}
	\]
	is the center (or quasi-kernel) of $\chi$. Equivalently, $\chi$ is of central type if and only if $\bar{\chi} \in \Irr(G/\ker(\chi))$ is fully ramified over $Z(G/\ker(\chi))$, where $\chi$ is the lift of $\bar{\chi}$ to $G$. In particular, $G/\ker(\chi)$ is a group of central type with faithful character $\bar{\chi}$ (see~\cite{LewisGVZ2}).
	
	Groups in which every irreducible complex character is of central type are called \emph{GVZ-groups}. These were first studied in~\cite{Onogroup} under the name \emph{groups of Ono type}. A conjugacy class $\mathcal{C}$ of $G$ is said to be of \emph{Ono type} if for every $\chi \in \Irr(G)$ and every $g \in \mathcal{C}$, either $\chi(g)=0$ or $|\chi(g)|=\chi(1)$. The group $G$ is of Ono type if all its conjugacy classes satisfy this condition. Such groups were first introduced by Ono~\cite{Ono}.
	
	A group $G$ is called \emph{nested} if for all $\chi, \psi \in \Irr(G)$, either $Z(\chi) \subseteq Z(\psi)$ or $Z(\psi) \subseteq Z(\chi)$. A GVZ-group $G$ is called a \emph{nested GVZ-group} if it is nested. In this case, $Z(\psi) \subseteq Z(\chi)$ whenever $\chi(1) \leq \psi(1)$ for $\chi, \psi \in \Irr(G)$ (see~\cite[Lemma 7.1]{LewisGVZ2}), and moreover,
	\[
	Z(\psi) \subset Z(\chi) \quad \Longleftrightarrow \quad \chi(1) < \psi(1), \quad \text{for } \chi, \psi \in \Irr(G).
	\]
	A group $G$ has an irreducible character $\chi$ with $\chi(1)=|G/Z(\chi)|^{\frac{1}{2}}$ if and only if $\chi(g)=0$ for all $g \in G \setminus Z(\chi)$ (see~\cite[Corollary 2.30]{I}). Hence, this formulation of nested GVZ-groups coincides with that used by Nenciu in~\cite{NenciuGVZ, NenciuGVZ2}, motivated by problems of Berkovich~\cite{Berkovich}. It is well known that every GVZ-group is nilpotent. Moreover, any nested GVZ-group is nilpotent and decomposes as the direct product of a nested GVZ $p$-group and an abelian group (see~\cite[Corollary~2.5]{NenciuGVZ}). A group $G$ is called \emph{flat} if 
	$|\cl_G(g)| = |\langle [g, x] : x \in G\rangle|$ holds for every element $g \in G$, where $\cl_G(g)$ represents the conjugacy class of $g$ in $G$. It is noteworthy that a group $G$ is flat if and only if it is a GVZ-group (see~\cite[Theorem~A]{Burkett}). For further background on these classes of groups, see~\cite{Burkett2, LewisGVZ, LewisGVZ3, Burkett, Ram, Ram6, FM, MLL, LewisGVZ2, NenciuGVZ, NenciuGVZ2}.
	
	A non-abelian group $G$ is called a \emph{VZ-group} if every $\chi \in \nl(G)$ is fully ramified over $Z(G)$. Such groups form a subclass of nested GVZ-groups and have nilpotency class $2$ (see~\cite{FM}). In fact, every nilpotent group of class $2$ is a GVZ-group (see~\cite[Theorem~2.31]{I}), and hence any nested group of class $2$ is a nested GVZ-group. In particular, all two-generator $p$-groups of class $2$ are nested GVZ $p$-groups (see~\cite{Nenciu2generators}), although they need not be VZ $p$-groups. There also exist nested GVZ $p$-groups of arbitrarily large nilpotency class. For each $n \geq 1$, Nenciu~\cite{NenciuGVZ2} constructed a family of nested GVZ $p$-groups of order $p^{2n+1}$, exponent $p$, and class $n+1$, where $p>n+1$ is prime. Similarly, for each $n \geq 1$, Lewis~\cite{LewisGVZ2} constructed such groups of exponent $p^{\,n+1}$ for odd primes $p$. In this article, we classify all GVZ $p$-groups (respectively, nested GVZ $p$-groups) of order at most $p^6$, where $p$ is an odd prime.
	
	In an earlier work with Prajapati~\cite{Ram6}, we showed that the properties of being a GVZ-group and a nested GVZ-group are invariant under isoclinism (the notion is introduced in Section~\ref{sec:preliminaries}).
	
	\begin{theorem}\cite[Theorem~4]{Ram6}\label{thm:isoGVZ}
		Let $G$ and $H$ be finite isoclinic groups. If $G$ is a GVZ-group (respectively, a nested GVZ-group), then so is $H$.
	\end{theorem}
	
	James~\cite{RJ} classified all $p$-groups of order at most $p^6$ for odd primes $p$, although some errors occur in the case of order $p^6$. Earlier, Bender~\cite{Bender} had determined the groups of order $p^5$ for odd primes $p$. An independent classification of groups of order $p^6$ was later obtained in~\cite{O'Brien}. For $p \geq 5$, the number of isomorphism types of groups of order $p^6$ is
	$$
	3p^2 + 39p + 344 + 24 \gcd(p - 1, 3) + 11 \gcd(p - 1, 4) + 2 \gcd(p - 1, 5),
	$$
	while for $p=3$ there are $504$ such groups (see~\cite[Theorem 1]{O'Brien}). Presentations of $p$-groups of order at most $p^5$, organized by isoclinism families, appear in~\cite{RJ}, and corrected presentations for order $p^6$ with $p \geq 7$ are given in~\cite{NewmanO`Brien}. The cases $p \in \{3,5\}$ are available in the \textsc{SmallGroups} library of {\sf GAP}~\cite{Gap} and {\sf MAGMA}~\cite{Magma}.
	
	The isoclinism family $\Phi_1$ consists of abelian groups of order $p^n$ for all $1 \leq n \leq 6$. Non-abelian groups of order $p^3$ lie in $\Phi_2$. Groups of order $p^4$ fall into three families; those of class $2$ lie in $\Phi_2$, while those of class $3$ lie in $\Phi_3$. There are $10$ isoclinism classes for groups of order $p^5$, denoted $\Phi_i$ for $1 \leq i \leq 10$, and $43$ isoclinism classes for groups of order $p^6$, denoted $\Phi_i$ for $1 \leq i \leq 43$.
	
	By Theorem~\ref{thm:isoGVZ}, the classification of GVZ-groups (respectively, nested GVZ-groups) of order $p^n$ reduces to determining the corresponding isoclinism families. We prove Theorem~\ref{thm:isoGVZp^6}, which classifies GVZ $p$-groups (respectively, nested GVZ $p$-groups) of order at most $p^6$.

		\begin{theorem}\label{thm:isoGVZp^6}
		Let $G$ be a finite $p$-group of order at most $p^6$, where $p$ is an odd prime. Then $G$ is a GVZ-group if and only if
		$$
		G \in \Phi_1 \cup \Phi_{2} \cup \Phi_{4} \cup \Phi_{5} \cup \Phi_{7} \cup \Phi_{8} \cup \Phi_{11} \cup \Phi_{12} \cup \Phi_{13} \cup \Phi_{14} \cup \Phi_{15} \cup \Phi_{18} \cup \Phi_{21}.
		$$
		Moreover, $G$ is a nested GVZ-group if and only if
		$$
		G \in \Phi_1 \cup \Phi_{2} \cup \Phi_{5} \cup \Phi_{7} \cup \Phi_{8} \cup \Phi_{13} \cup \Phi_{14} \cup \Phi_{15} \cup \Phi_{21}.$$
	\end{theorem}
	
Lewis~\cite{LewisGVZ2}, using {\sf MAGMA}~\cite{Magma}, determined that there are exactly $111$ non-abelian nested GVZ-groups of order $3^6$. As a consequence of Theorem~\ref{thm:isoGVZp^6}, Corollary~\ref{coro:GVZcounting} provides explicit formulas for the number of isomorphism classes of GVZ $p$-groups (respectively, nested GVZ $p$-groups) of order $p^6$ for all primes $p \geq 5$.
	\begin{corollary}\label{coro:GVZcounting}
		Let $p \geq 5$ be a prime. Then the number of isomorphism types of GVZ $p$-groups of order $p^6$ is
		$$
		\frac{3p^2 + 28p + 315 + 2 \gcd(p - 1, 3) + 2 \gcd(p - 1, 4)}{2}.
		$$
		Moreover, the number of isomorphism types of nested GVZ $p$-groups of order $p^6$ is
		$$
		\frac{3p^2 + 10p + 187}{2}.
		$$
	\end{corollary}

	\section{Preliminaries}\label{sec:preliminaries}
	
	In this section, we fix notation and recall some basic prerequisites. Throughout, $p$ denotes an odd prime and $G$ a finite group. Let $\Irr(G)$ be the set of irreducible complex characters of $G$, with $\lin(G)=\{\chi \in \Irr(G) : \chi(1)=1\}$ and $\nl(G)=\{\chi \in \Irr(G) : \chi(1)\neq 1\}$. For $m \in \mathbb{Z}_{>0}$, set $\Irr_m(G)=\{\chi \in \Irr(G) : \chi(1)=m\}$ and $\cd(G)=\{\chi(1) : \chi \in \Irr(G)\}$. If $N \trianglelefteq G$, write $\Irr(G \mid N)=\{\chi \in \Irr(G) : N \nsubseteq \ker(\chi)\}$ and $\nl(G \mid N)=\{\chi \in \nl(G) : N \nsubseteq \ker(\chi)\}$. For $\chi \in \Irr(G)$ and $H \leq G$, denote by $\chi \downarrow_H$ the restriction of $\chi$ to $H$. Further, let $\nil(G)$ and $\mu(G)$ denote the nilpotency class and the minimal faithful permutation degree of $G$, respectively. All other notation is standard.
	
	We now recall some preliminary notions, beginning with the definition of a \emph{Camina pair} (see \cite{Camina}).
	\begin{definition}
		Let $N$ be a normal subgroup of a finite group $G$. The pair $(G, N)$ is called a \emph{Camina pair} if $1 < N < G$ and, for every $g \in G \setminus N$, the element $g$ is conjugate to each element of the coset $gN$.
	\end{definition}
	
	A necessary and sufficient condition for the pair $(G, N)$ to be a Camina pair is that $\chi(g)=0$ for all $g \in G \setminus N$ and every $\chi \in \Irr(G \mid N)$. It is straightforward to verify that if $(G, N)$ is a Camina pair, then $Z(G) \leq N \leq G'$. In \cite{MLL3}, Lewis first studied groups $G$ for which $(G, Z(G))$ forms a Camina pair and showed that such a group must be a $p$-group for some prime $p$. The following lemma describes the relationship between $\Irr(G \mid Z(G))$ and $\Irr(Z(G))$ in this situation.
	
	\begin{lemma}\label{lemma:Caminacharacter}\textnormal{\cite[Lemma 3.3]{SKP}}
		Let $(G, Z(G))$ be a Camina pair. Then there exists a bijection between the sets $\Irr(G \mid Z(G))$ and $\Irr(Z(G)) \setminus \{1_{Z(G)}\}$, where $1_{Z(G)}$ denotes the trivial character of $Z(G)$. For each $1_{Z(G)} \neq \mu \in \Irr(Z(G))$, the corresponding character $\chi_\mu \in \nl(G)$ is given by
		\begin{equation}\label{Caminacharacter}
			\chi_\mu(g) =
			\begin{cases}
				|G/Z(G)|^{\frac{1}{2}} \, \mu(g) & \text{if } g \in Z(G), \\
				0 & \text{otherwise}.
			\end{cases}
		\end{equation}
	\end{lemma}
	
	A pair $(G, N)$ is called a \emph{generalized Camina pair} if $N$ is a normal subgroup of $G$ and every nonlinear irreducible complex character of $G$ vanishes on $G \setminus N$ (see \cite{MLL}). A group $G$ is called a \emph{VZ-group} if $(G, Z(G))$ is a generalized Camina pair.
	
	We now recall the following lemmas concerning the degrees and vanishing subgroups of irreducible characters of a group, which will be used frequently in the sequel.
	\begin{lemma}\cite[Theorem 20]{Berkovich}\label{lemma:Berkovich}
		If $G$ is a finite $p$-group, then $\chi(1)^2$ divides $|G/Z(G)|$ for each $\chi \in \Irr(G)$.
	\end{lemma}
	
	\begin{lemma}\label{lemma:innerproductrestriction}\textnormal{\cite[Lemma 2.29]{I}}
		Let $H$ be a subgroup of a finite group $G$, and let $\chi$ be a character of $G$. Then we have
		\[
		\langle \chi \downarrow_H, \chi \downarrow_H \rangle \leq |G/H| \langle \chi, \chi \rangle,
		\]
		with equality if and only if $\chi(g)=0$ for all $g \in G \setminus H$.
	\end{lemma}
	
	We conclude this section by recalling the notion of \emph{isoclinism}. This notion, introduced by Hall in \cite{PH} for the classification of $p$-groups, is ubiquitous throughout this article.
	
	\begin{definition}
		Two finite groups $G$ and $H$ are said to be \emph{isoclinic} if there exist isomorphisms $\theta : G/Z(G) \to H/Z(H)$ and $\phi : G' \to H'$ such that the diagram
		\[
		\begin{tikzcd}
			G/Z(G) \times G/Z(G) \arrow{d}{\theta \times \theta} \arrow{r}{a_G}
			& G' \arrow{d}{\phi} \\
			H/Z(H) \times H/Z(H) \arrow{r}{a_H}
			& H'
		\end{tikzcd}
		\]
		is commutative, where $a_G(g_1Z(G), g_2Z(G)) = [g_1,g_2]$ for $g_1,g_2 \in G$, and $a_H(h_1Z(H), h_2Z(H)) = [h_1,h_2]$ for $h_1,h_2 \in H$.
	\end{definition}
	
	The pair $(\theta, \phi)$ is called an \emph{isoclinism} from $G$ onto $H$. This notion generalizes isomorphism, and it is well known that isoclinic nilpotent groups have the same nilpotency class.

	\section{GVZ-groups of order $p^6$}
	In this section, we classify all GVZ-groups (respectively, nested GVZ-groups) of order $p^6$. Hall~\cite{PH} introduced isoclinism as a generalization of isomorphism for the classification of $p$-groups, and James~\cite{RJ} later employed this notion to classify $p$-groups of order up to $p^6$. The groups of order $p^6$ are partitioned into $43$ isoclinism families, denoted by $\Phi_i$ for $1 \leq i \leq 43$, where the family $\Phi_1$ consists of all abelian groups. Although some inaccuracies occur in the classification given in~\cite{RJ}, the subsequent works of~\cite{O'Brien, NewmanO`Brien} refine and confirm the structure of these isoclinism families. The invariants associated with these families are summarized in~\cite[Subsection~4.1]{RJ}. In our analysis, we primarily use invariants such as the structure of the derived subgroup, the central quotient, and the character degrees. These invariants suffice to determine whether all groups in a given isoclinism family are GVZ-groups (respectively, nested GVZ-groups). When explicit group presentations are required, we adopt the descriptions provided in~\cite{NewmanO`Brien}. We now proceed with Lemma~\ref{lemma:isoGVZp^5}.
	
	\begin{lemma}\cite[Corollary~5]{Ram6}\label{lemma:isoGVZp^5}
		Let $G$ be a non-abelian group of order $p^5$. Then $G$ is a GVZ-group if and only if $G \in \Phi_{2} \cup \Phi_{4} \cup \Phi_{5} \cup \Phi_{7} \cup \Phi_{8}$. Moreover, $G$ is a nested GVZ-group if and only if $G \in \Phi_{2} \cup \Phi_{5} \cup \Phi_{7} \cup \Phi_{8}$.	
	\end{lemma}
	
	Lemma~\ref{lemma:direct product} describes the behavior of direct products of GVZ-groups and nested GVZ-groups.
	\begin{lemma}\label{lemma:direct product}
		Let $G$ and $H$ be finite GVZ-groups. Then the direct product $G \times H$ is also a GVZ-group. Furthermore, if $G$ and $H$ are finite nested GVZ-groups, then $G \times H$ is a nested GVZ-group if and only at least one of $G$ and $H$ is abelian.
	\end{lemma}
	
	\begin{proof}
		Let $\chi \in \Irr(G)$ and $\psi \in \Irr(H)$. Recall that
		\[
		\Irr(G \times H) = \{ \chi \otimes \psi : \chi \in \Irr(G), \psi \in \Irr(H) \}.
		\]
		Note that
		\[
		Z(\chi \otimes \psi) = Z(\chi) \times Z(\psi),
		\]
		where $Z(\chi) = \{ g \in G : |\chi(g)| = \chi(1) \}$ and $Z(\psi) = \{ g \in G : |\psi(g)| = \psi(1) \}$.
		
		Since $G$ and $H$ are GVZ-groups, $\chi$ vanishes outside $Z(\chi)$ and $\psi$ vanishes outside $Z(\psi)$. Then for all $(g,h) \in G \times H \setminus Z(\chi \otimes \psi)$, we have
		\[
		(\chi \otimes \psi)(g,h) = \chi(g)\psi(h) = 0,
		\]
		so $G \times H$ is a GVZ-group.
		
		Next, let $G$ and $H$ be finite nested GVZ-groups. Without loss of generality suppose that $H$ is abelian. Hence, observe that $G$ and $G \times H$ are isoclinic. Therefore, by Theorem~\ref{thm:isoGVZ}, $G \times H$ is a nested GVZ-group. For the converse part, suppose that both the groups $G$ and $H$ are non-abelian. Let $\chi \in \nl(G)$ and $\psi \in \nl(H)$. Then there exist proper subgroups $G_1 < G$ and $H_1 < H$ such that
		\[ Z(\chi) = G_1 \, \, \, \, \text{and} \, \, \, \, Z(\psi) = H_1.\]
		Consider the characters of $G \times H$
		\[
		\chi \otimes 1_H \, \, \, \, \text{and} \, \, \, \, 1_G \otimes \psi,
		\]
		where $1_G$ and $1_H$ denote the trivial characters of $G$ and $H$, respectively. Furthermore, we have
		\[
		Z(\chi \otimes 1_H) = G_1 \times H \, \, \, \, \text{and} \, \, \, \,
		Z(1_G \otimes \psi) = G \times H_1.
		\]
		Therefore, neither
		\[
		G_1 \times H \subseteq G \times H_1 \, \, \, \, \text{nor} \, \, \, \,
		G \times H_1 \subseteq G_1 \times H
		\]
		holds. Thus, $G \times H$ is not a nested GVZ-group. This completes the proof of Lemma~\ref{lemma:direct product}.
	\end{proof}
	
	We now prove Lemma~\ref{lemma:class1}, which classifies all GVZ-groups (respectively, nested GVZ-groups) of order $p^6$ lying in $\bigcup_{i=1}^{10} \Phi_i$.
	\begin{lemma}\label{lemma:class1}
	Let $G$ be a group of order $p^6$ such that $G \in \bigcup_{i=1}^{10} \Phi_i$. Then $G$ is a GVZ-group if and only if $G \in \Phi_{1} \cup \Phi_{2} \cup \Phi_{4} \cup \Phi_{5} \cup \Phi_{7} \cup \Phi_{8}$. Moreover, $G$ is a nested GVZ-group if and only if $G \in \Phi_{1} \cup \Phi_{2} \cup \Phi_{5} \cup \Phi_{7} \cup \Phi_{8}$.
	\end{lemma}
	\begin{proof}
		Note that, by definition of GVZ-groups, every abelian group is a GVZ-group and, in fact, a nested GVZ-group. Hence, by Lemma~\ref{lemma:isoGVZp^5}, all groups of order $p^5$ belonging to $\Phi_{1} \cup \Phi_{2} \cup \Phi_{4} \cup \Phi_{5} \cup \Phi_{7} \cup \Phi_{8}$ are GVZ-groups, and those belonging to $\Phi_{1} \cup \Phi_{2} \cup \Phi_{5} \cup \Phi_{7} \cup \Phi_{8}$ are nested GVZ-groups. Moreover, these conditions are both necessary and sufficient.
		
		Furthermore, for each $i \in \{1,2,\ldots,10\}$, there exists a group $G$ of order $p^6$ in the isoclinism family $\Phi_i$ such that $G \cong H \times C_p$, where $H$ is a group of order $p^5$ in $\Phi_i$ and $C_p$ is the cyclic group of order $p$. Therefore, by Lemma~\ref{lemma:direct product} and Theorem~\ref{thm:isoGVZ}, the desired result follows. This completes the proof of Lemma~\ref{lemma:class1}.
	\end{proof}

	Before proving Lemma~\ref{lemma:class2}, we establish the following facts, which are essential for the proof of this lemma as well as the subsequent lemmas.
\begin{lemma}\label{lemma:Special}
	Let $G$ be a finite group. Suppose $G$ has an irreducible character $\chi \in \Irr(G)$ such that $\chi(1)= |G/Z(G)|^{\frac{1}{2}}$. Then $\chi$ is of central type and $Z(\chi)=Z(G)$.
\end{lemma}
\begin{proof}
	Observe that $\chi \downarrow_{Z(G)}=p^2\mu$ for some $\mu \in \Irr(Z(G))$. Furthermore, we have 
	\begin{align*}
		\langle \chi \downarrow_{Z(G)}, \chi \downarrow_{Z(G)} \rangle = & \langle |G/Z(G)|^{\frac{1}{2}}\mu, |G/Z(G)|^{\frac{1}{2}}\mu \rangle\\
		= & |G/Z(G)| \langle \mu, \mu \rangle\\
		= & |G/Z(G)|^{\frac{1}{2}}\\
		= & |G/Z(G)|\langle \chi, \chi \rangle.
	\end{align*}
	Hence, from Lemma \ref{lemma:innerproductrestriction}, we have $\chi(g)=0$ for all $g \in G \setminus Z(G)$. Therefore, $(G, Z(G))$ is a Camina pair. Hence, $\chi$ is of central type and $Z(\chi)=Z(G)$. This completes the proof of Lemma~\ref{lemma:Special}.
\end{proof}

\begin{lemma}\label{lemma:qotient}
	Let $N \trianglelefteq G$ and let $\bar{\chi} \in \Irr(G/N)$. Let $\chi \in \Irr(G)$ be the lift of $\bar{\chi}$ to $G$, i.e., $\chi(g)=\bar{\chi}(gN)$ for all $g \in G$. Then
	\[
	\chi \text{ is of central type in } G \;\Longleftrightarrow\; \bar{\chi} \text{ is of central type in } G/N.
	\]
\end{lemma}

\begin{proof}
Since $\chi(g)=\bar{\chi}(gN)$, we have $|\chi(g)| = |\bar{\chi}(gN)|$ for all $g \in G$. Hence
\[
Z(\chi)
= \{ g \in G : |\chi(g)| = \chi(1) \}
= \{ g \in G : |\bar{\chi}(gN)| = \bar{\chi}(1) \}
= \pi^{-1}(Z(\bar{\chi})),
\]
where $\pi : G \to G/N$ is the natural projection.

Suppose that $\chi$ is of central type. Let $xN \notin Z(\bar{\chi})$. Then $x \notin Z(\chi)$, so $\chi(x)=0$. Hence $\bar{\chi}(xN)=0$, and thus $\bar{\chi}$ is of central type.

Conversely, suppose that $\bar{\chi}$ is of central type. Let $g \notin Z(\chi)$. Then $gN \notin Z(\bar{\chi})$, so $\bar{\chi}(gN)=0$. Consequently, $\chi(g)=\bar{\chi}(gN)=0$, and therefore $\chi$ is of central type.
\end{proof}

Lemma~\ref{lemma:Phi13} describes a technique for determining all irreducible complex characters of groups of order $p^6$ belonging to $\Phi_{12} \cup \Phi_{13}$. This lemma is essential for the proof of Lemma~\ref{lemma:class2}.
\begin{lemma}\label{lemma:Phi13}
	Let $G$ be a group of order $p^6$ such that $G \in \Phi_{12} \cup \Phi_{13}$. Then we have the following.
	\begin{enumerate}
		\item $\cd(G) = \{1, p, p^2\}$.
		\item There is a bijection between the sets $\{\bar{\chi} \in \nl(G/K) : C_p \cong K < Z(G)\}$ and $\nl(G)$, where for $\chi \in \nl(G)$, $\bar{\chi}$ lifts to $\chi$.
		\item If $C_p \cong K < Z(G)$, then $(G/K, Z(G/K))$ is a generalized Camina pair.
	\end{enumerate}
\end{lemma}
\begin{proof}
	The part (1) follows directly from \cite[Section 4.1]{RJ}. Let $G$ be a finite $p$-group such that $Z(G)$ is elementary abelian with $|Z(G)| > p$. Then every $\chi \in \Irr(G)$ is the lift of some $\overline{\chi} \in \Irr(G/K)$, where $K \cong C_p$ and $K < Z(G)$. Hence, part (2) follows from the fact that $Z(G) = G' \cong C_p \times C_p$.
	
	Next, let $G$ be a group of order $p^6$ such that $G \in \Phi_{12} \cup \Phi_{13}$, and suppose that $Z(G) = G' \cong \langle a, b \rangle$. Let $K \cong C_p$ with $K < Z(G)$. Then $K = \langle a \rangle$ or $K = \langle a^i b \rangle$ for some $0 \leq i \leq p-1$. If $K = \langle a \rangle$ or $K = \langle b \rangle$, then $(G/K)' \cong \langle bK \rangle \cong C_p$ or $(G/K)' \cong \langle aK \rangle \cong C_p$, respectively. If $K = \langle a^i b \rangle$ for $1 \leq i \leq p-1$, then $(G/K)' \cong \langle aK \rangle \cong C_p$, since $bK = a^{-i}K$. Thus, in all cases, $(G/K)' \cong C_p$. Therefore, $G/K$ is a group of order $p^5$ with derived subgroup of order $p$. It follows that $(G/K, Z(G/K))$ is a generalized Camina pair (see \cite[Section 4.1]{RJ}). Hence, part (3) follows. This completes the proof of Lemma~\ref{lemma:Phi13}.
\end{proof}
	
	We are now ready to prove Lemma~\ref{lemma:class2}, which shows that every group of order $p^6$ in $\bigcup_{i=11}^{15} \Phi_i$ is a GVZ-group, and furthermore gives a classification of all nested GVZ-groups contained in $\bigcup_{i=11}^{15} \Phi_i$.
	\begin{lemma}\label{lemma:class2}
		Let $G$ be a group of order $p^6$ such that $G \in \bigcup_{i=11}^{15} \Phi_i$. Then $G$ is a GVZ-group. Furthermore, $G$ is a nested GVZ-group if and only if $G \in \Phi_{13} \cup \Phi_{14} \cup \Phi_{15}$.
	\end{lemma}
	\begin{proof}
		Recall that a nilpotent group of class $2$ is a GVZ-group (see~\cite[Theorem~2.31]{I}). The groups of order $p^6$ belonging to $\Phi_{11} \cup \Phi_{12} \cup \Phi_{13} \cup \Phi_{14} \cup \Phi_{15}$ have nilpotency class $2$ (see~\cite[Subsection~4.1]{RJ}). Hence, all these groups are GVZ-groups.
		
		Note that $|G/Z(G)|^{1/2} = p^2$ for $G \in \Phi_{12} \cup \Phi_{13} \cup \Phi_{14} \cup \Phi_{15}$. Therefore, by Lemma~\ref{lemma:Special}, every $\chi \in \Irr_{p^2}(G)$ is of central type and satisfies $Z(\chi)=Z(G)$. 
		
		If $G \in \Phi_{15}$, then $\cd(G)=\{1, p^2\}$ (see~\cite[Subsection~4.1]{RJ}). Thus $(G, Z(G))$ is a generalized Camina pair, and hence $G$ is a VZ-group. Consequently, $G \in \Phi_{15}$ is a nested GVZ-group.
		
		For $G \in \Phi_{11}$, we have $|\cd(G)|=2$ (see~\cite[Subsection~4.1]{RJ}). A nested GVZ $p$-group with exactly two character degrees is necessarily a VZ-group. Since this is not the case here, it follows that $G \in \Phi_{11}$ is not a nested GVZ-group.
		
		We now show that $G \in \Phi_{12}$ is not a nested GVZ-group. There exist groups in $\Phi_{12}$ of the form $G = H \times K$, where $H$ and $K$ are non-abelian nested GVZ-groups of order $p^3$. For $p \geq 7$, by~\cite{NewmanO`Brien}, we have
		\begin{align*}
			G_{(12,1)} &= \langle \alpha_1, \alpha_2, \alpha_3, \alpha_4, \alpha_5, \alpha_6 : [\alpha_3,\alpha_4]=\alpha_1,\; [\alpha_5,\alpha_6]=\alpha_2, \\
			&\qquad \alpha_1^p=\alpha_2^p=\alpha_3^p=\alpha_4^p=\alpha_5^p=\alpha_6^p=1 \rangle \\
			&\cong \langle \alpha_1,\alpha_3,\alpha_4 \rangle \times \langle \alpha_2,\alpha_5,\alpha_6 \rangle,
		\end{align*}
		where each factor is an extraspecial group of order $p^3$ and exponent $p$. Since extraspecial $p$-groups are VZ-groups, Lemma~\ref{lemma:direct product} implies that $G_{(12,1)}$ is not a GVZ-group. Hence, by Theorem~\ref{thm:isoGVZ}, no group in $\Phi_{12}$ is a nested GVZ-group.
		
		Next, consider $G \in \Phi_{14}$. Groups in this isoclinism family are two-generator $p$-groups. In particular, for $p \geq 7$, we focus on the group $G_{(14,3)} \in \Phi_{14}$ as described in~\cite{NewmanO`Brien}.
		\begin{align*}
			G_{(14, 3)} = &\langle \alpha_1, \alpha_2, \alpha_3, \alpha_4, \alpha_5, \alpha_6 : [\alpha_4, \alpha_6]= \alpha_2, [\alpha_3, \alpha_6]=[\alpha_4, \alpha_5]=\alpha_1,\\
			& \alpha_2^p= \alpha_1, \alpha_3^p= \alpha_2, \alpha_4^p= \alpha_3, \alpha_6^p= \alpha_5, \alpha_1^p= \alpha_5^p=1 \rangle\\
			=& \langle \alpha_4, \alpha_6 \rangle.
		\end{align*}
		Two-generator $p$-groups of class $2$ are nested GVZ-groups (see~\cite{Nenciu2generators}). Since groups in $\Phi_{14}$ have class $2$, it follows that $G_{(14,3)}$ is a nested GVZ-group. Hence, by Theorem~\ref{thm:isoGVZ}, every $G \in \Phi_{14}$ is a nested GVZ-group.
		
		Finally, let $G \in \Phi_{13}$. There is a bijection between the sets $\{\bar{\chi} \in \nl(G/K) : C_p \cong K < Z(G)\}$ and $\nl(G)$, where each $\bar{\chi}$ lifts to $\chi \in \nl(G)$. Moreover, $G/K$ is a VZ-group of order $p^5$ (see Lemma~\ref{lemma:Phi13}), and $\cd(G)=\{1,p,p^2\}$. Hence there exists a subgroup $K \cong C_p$ with $K < Z(G)$ such that $G/K$ is a VZ-group of order $p^5$ with $\cd(G/K)=\{1,p\}$. In this case, $\nl(G/K)=p^3 - p^2 = \Irr_p(G)$ (see~\cite[Subsection~4.1]{RJ}), and such a subgroup $K$ is unique. Let $Z(G/K)=H/K$. By Lemma~\ref{lemma:qotient}, every $\chi \in \Irr_p(G)$ is of central type with $Z(\chi)=H$. On the other hand, $\chi \in \Irr_{p^2}(G)$ is of central type with $Z(\chi)=Z(G)$, and we have $Z(G) < H$. Therefore, $G \in \Phi_{13}$ is a nested GVZ-group. This completes the proof of Lemma~\ref{lemma:class2}.
	\end{proof}
	
	Next, we prove Lemma~\ref{lemma:class3}, which shows that $G \in \Phi_{16} \bigcup_{i=22}^{43} \Phi_i$ is not a GVZ-group. Before proving the lemma, we recall some results related to GVZ-groups that will be used in the proof of Lemma~\ref{lemma:class3}. We begin with Lemma~\ref{lemma:nilclassGVZ}.
		\begin{lemma}\cite[Theorem~B]{Burkett}\label{lemma:nilclassGVZ}
		If $G$ is a GVZ-group, then $\nil(G) \leq |\cd(G)|$.
	\end{lemma}
	
	The minimal faithful permutation degree, denoted by $\mu(G)$, of a finite group $G$ is the least positive integer $n$ such that $G$ is isomorphic to a subgroup of the symmetric group $S_n$. Equivalently, $\mu(G)$ is the minimal degree of a faithful permutation representation of $G$, or, in other words, the smallest $n$ for which $G$ admits an embedding into $S_n$.
	
	\begin{lemma}\cite[Corollary~33]{Ayush2}\label{lemma:PermutationRep}
		Let $G$ be a GVZ $p$-group with cyclic center, where $p$ is an odd prime. Then $\mu(G)=|G/Z(G)|^{\frac{1}{2}}|Z(G)|$.
	\end{lemma}

	\begin{lemma}\label{lemma:class3}
		Let $G$ be a group of order $p^6$ such that $G \in \Phi_{16} \bigcup_{i=22}^{43} \Phi_i$. Then $G$ is not a GVZ-group.
	\end{lemma}
	\begin{proof}
		Let $G$ be a group of order $p^6$. Then observe that $|\cd(G)| \leq 3$. Furthermore, if $G \in \bigcup_{i \in \tau} \Phi_i$, where $\tau = \{23, 24, 25, \ldots, 43\} \setminus \{31, 32, 33, 34\}$, then $\nil(G) \in \{4, 5\}$ (see \cite[Subsection~4.1]{RJ}). Moreover, $\nil(G)=3$ and $|\cd(G)|=2$ for $G \in \Phi_{16}$ (see \cite[Subsection~4.1]{RJ}). Hence, Lemma~\ref{lemma:nilclassGVZ} implies that if $G \in \bigcup_{i \in \tau \cup \{16\}} \Phi_i$, then $G$ is not a GVZ $p$-group.
		
		Next, suppose to the contrary that $G$ is a group of order $p^6$ such that 
		$G \in \Phi_{22} \cup \Phi_{31} \cup \Phi_{32} \cup \Phi_{33} \cup \Phi_{34}$ 
		and $G$ is a GVZ $p$-group. Note that $Z(G) \cong C_p$ (see \cite[Subsection~4.1]{RJ}). 
		Thus, $G$ is a GVZ $p$-group with cyclic center. Hence, by Lemma~\ref{lemma:PermutationRep}, we obtain
		$$
		\mu(G) = |G/Z(G)|^{\frac{1}{2}}\,|Z(G)| = p^{\frac{7}{2}}.
		$$
		This implies that $\mu(G)$ is not a positive integer, which contradicts the definition of $\mu(G)$ (see \cite{Ayush1} for the exact value of $\mu(G)$). Therefore, 
		$G \in \Phi_{22} \cup \Phi_{31} \cup \Phi_{32} \cup \Phi_{33} \cup \Phi_{34}$ is not a GVZ $p$-group. This completes the proof of Lemma~\ref{lemma:class3}.
	\end{proof}
	
Lemma~\ref{lemma:Camina} provides a method for determining all irreducible complex characters of groups of order $p^6$ that belong to $\Phi_{21}$. This result plays a crucial role in the proof of the subsequent lemma.
	\begin{lemma}\label{lemma:Camina}
		Let $G$ be a group of order $p^6$ with $G \in \Phi_{21}$. Then we have the following.
		\begin{enumerate}
			\item $\cd(G) = \{1, p, p^2\}$.
			
			\item There is a bijection between the sets $\Irr_{p}(G)$ and $\nl(G/Z(G))$. In particular, every irreducible complex character of $G$ of degree $p$ is the lift of a non-linear irreducible complex character of $G/Z(G)$.
			
			\item The pair $(G, Z(G))$ is a Camina pair. Moreover, there is a bijection between the sets $\Irr_{p^2}(G)$ and $\Irr(Z(G)) \setminus \{1_{Z(G)}\}$, where $1_{Z(G)}$ denotes the trivial character of $Z(G)$.
		\end{enumerate}
	\end{lemma}
	\begin{proof}
	Suppose $G$ is a group of order $p^6$ with $G \in \Phi_{21}$.
	\begin{enumerate}
		\item This follows immediately (see \cite[Section 4.1]{RJ}).
		
		\item Since $G/Z(G) \cong \Phi_2(1^4)$ for $G \in \Phi_{21}$ (see \cite[Section 4.1]{RJ}), it follows that $G/Z(G)$ is a VZ $p$-group of order $p^4$. In particular, $G$ is of nilpotency class $2$. Moreover, we have
		\[
		|\nl(G/Z(G))| = p^2 - p = |\Irr_p(G)|.
		\]
		Hence, there is a bijection between $\Irr_p(G)$ and $\nl(G/Z(G))$, and every irreducible character of $G$ of degree $p$ arises as the lift of a non-linear irreducible character of $G/Z(G)$.
		
		\item Note that $|G/Z(G)|^{1/2} = p^2$ for $G \in \Phi_{21}$. By Lemma~\ref{lemma:Special}, every $\chi \in \Irr_{p^2}(G)$ is of central type and satisfies $Z(\chi) = Z(G)$. Let $\chi \in \Irr(G \mid Z(G))$. Since $Z(G) \subseteq G'$, part (2) implies that $\chi \in \Irr_{p^2}(G)$. Furthermore, we have
		\[
		|\Irr(Z(G)) \setminus \{1_{Z(G)}\}| = p^2 - 1 = |\Irr_{p^2}(G)|.
		\]
		This establishes a bijection between $\Irr_{p^2}(G)$ and $\Irr(Z(G)) \setminus \{1_{Z(G)}\}$, and hence $(G, Z(G))$ is a Camina pair. The result now follows from Lemma~\ref{lemma:Caminacharacter}. \qedhere
	\end{enumerate}
	\end{proof}
	
	Lemma~\ref{lemma:class4} shows that every group of order $p^6$ belonging to $\Phi_{21}$ is a nested GVZ $p$-group.
	\begin{lemma}\label{lemma:class4}
		Let $G$ be a group of order $p^6$ such that $G \in \Phi_{21}$. Then $G$ is a nested GVZ-group.
	\end{lemma}
	\begin{proof}
			Let $G$ be a group of order $p^6$ with $G \in \Phi_{21}$. Then $\cd(G)=\{1, p, p^2\}$, $Z(G)\cong C_p \times C_p$, and $G/Z(G)\cong \Phi_{2}(1^4)$ (see \cite[Subsection~4.1]{RJ}). By Lemma~\ref{lemma:Camina}, there is a bijection between the sets $\Irr_{p}(G)$ and $\nl(G/Z(G))$. Let $\chi \in \Irr_{p}(G)$ be the lift of $\bar{\chi} \in \nl(G/Z(G))$. Since $G/Z(G)$ is a VZ-group, let $Z(G/Z(G))=H/Z(G)$. Then, by Lemma~\ref{lemma:qotient}, $\chi$ is of central type and satisfies $Z(\chi)=H$.
			
			Furthermore, by Lemma~\ref{lemma:Camina}, the pair $(G, Z(G))$ is a Camina pair. Hence, there is a bijection between the sets $\Irr_{p^2}(G)$ and $\Irr(Z(G)) \setminus \{1_{Z(G)}\}$. For each $1_{Z(G)} \neq \mu \in \Irr(Z(G))$, the corresponding character $\chi_\mu \in \Irr_{p^2}(G)$ is given by
			\[
			\chi_\mu(g) = 
			\begin{cases}
				p^2 \mu(g) & \text{if } g \in Z(G),\\
				0          & \text{otherwise.}
			\end{cases}
			\]
			Thus, we have $Z(\chi_\mu)=Z(G)$, and hence
			\[
			|G/Z(\chi_\mu)|^{\frac{1}{2}} = p^2 = \chi_\mu(1).
			\]
			Therefore, each $\chi_\mu \in \Irr_{p^2}(G)$ is of central type. Moreover, note that $|Z(\chi_\mu)| = |Z(G)| = p^2$ for all $\chi_\mu \in \Irr_{p^2}(G)$, whereas $|Z(\chi)| = p^4$ for all $\chi \in \Irr_{p}(G)$. This shows that $G$ is strictly nested by degrees. Hence, $G$ is a nested GVZ $p$-group. This completes the proof of Lemma~\ref{lemma:class4}.
	\end{proof}
	
	Finally, we conclude this section by proving Lemma~\ref{lemma:class5}, which classifies all GVZ-groups (respectively, nested GVZ-groups) among the remaining isoclinic families. Prior to this, we establish Lemma~\ref{lemma:phi17to20}, which provides a method for determining all irreducible complex characters of groups of order $p^6$ belonging to $\bigcup_{i=17}^{20} \Phi_i$.
	\begin{lemma}\label{lemma:phi17to20}
		Let $G$ be a group of order $p^6$ such that $G \in \Phi_{17} \cup \Phi_{18} \cup \Phi_{19} \cup \Phi_{20}$. Then we have the following.
		\begin{enumerate}
			\item $\cd(G) = \{1, p, p^2\}$.
			\item Each $\chi \in \nl(G)$ is the lift of some $\bar{\chi} \in \nl(G/K)$, where $C_p \cong K < Z(G)$.
			\item There is a bijection between the sets $\{\bar{\chi} \in \nl(G/K \mid Z(G)/K) : C_p \cong K < Z(G)\}$ and $\nl(G\mid Z(G))$, where for $\chi \in \nl(G)$, $\bar{\chi}$ lifts to $\chi$.
		\end{enumerate}
	\end{lemma}
	\begin{proof}
	Suppose $G$ is a group of order $p^6$ such that $G \in \Phi_{17} \cup \Phi_{18} \cup \Phi_{19} \cup \Phi_{20}$. Part (1) follows immediately (see \cite[Section 4.1]{RJ}). 
	
	Furthermore, we have $Z(G) \cong C_p \times C_p$ with $Z(G) < G'$. Since $Z(G)$ is not cyclic, there exists a subgroup $K \cong C_p$ with $K < Z(G)$ such that $K \subseteq \ker(\chi)$ for all $\chi \in \nl(G)$. Hence, each $\chi \in \nl(G)$ is the lift of some $\bar{\chi} \in \nl(G/K)$ for some subgroup $K \cong C_p$ with $K < Z(G)$. Therefore, part (2) also follows.
	
	Now assume that $\chi \in \nl(G)$ is such that $Z(G) \subseteq \ker(\chi)$. Then $\chi$ is the lift of a non-linear character of $G/Z(G)$. Consequently, we obtain
	\[
	\left|\left\{ \bar{\chi} \in \nl(G/K \mid Z(G)/K) \,:\, K \cong C_p,\; K < Z(G) \right\}\right| 
	= |\nl(G \mid Z(G))|.
	\]
	This completes the proof.
	\end{proof} 
	
	Lemma~\ref{lemma:class5} classifies all GVZ-groups and nested GVZ-groups of order $p^6$ belonging to $\bigcup_{i=17}^{20} \Phi_i$.
	\begin{lemma}\label{lemma:class5}
		Let $G$ be a group of order $p^6$ such that $G \in \Phi_{17} \cup \Phi_{18} \cup \Phi_{19} \cup \Phi_{20}$. Then $G$ is not a GVZ-group for $G \in \Phi_{17} \cup \Phi_{19} \cup \Phi_{20}$. Furthermore, if $G \in \Phi_{18}$, then $G$ is a GVZ-group but not a nested GVZ-group.
	\end{lemma}
	\begin{proof}
	For this proof, we consider a representative from each isoclinic family and analyze its classification to determine whether the groups in that family are GVZ $p$-groups, nested GVZ $p$-groups, or neither. We use the presentation of the representative group as given in \cite{NewmanO`Brien}. Let $G$ be a group of order $p^6$ such that $G \in \Phi_{17} \cup \Phi_{18} \cup \Phi_{19} \cup \Phi_{20}$.
		
		First, consider $G \in \Phi_{17}$. We show that $G$ has an irreducible character $\chi \in \Irr(G)$ that is not of central type. For instance, we have
		\begin{align*}
			G=G_{(17, 1)}=&\langle \alpha_1, \alpha_2, \alpha_3, \alpha_4, \alpha_5, \alpha_6 : [\alpha_5, \alpha_6]= \alpha_3, [\alpha_4, \alpha_5]=\alpha_2, [\alpha_3, \alpha_6]=\alpha_1,\\ &\alpha_1^p=\alpha_2^p=\alpha_3^p=\alpha_4^p=\alpha_5^p= \alpha_6^p=1 \rangle,
		\end{align*}
		where $p \geq 7$. Then we have $Z(G) = \langle \alpha_1, \alpha_2 \rangle \cong C_p \times C_p$ and $G' = \langle \alpha_1, \alpha_2, \alpha_3 \rangle \cong C_p \times C_p \times C_p$.
		Let $K = \langle \alpha_2 \rangle \cong C_p$. Then we get
		\[
		(G/K)' = \langle \alpha_1 Z(G), \alpha_3 Z(G) \rangle \cong C_p \times C_p
		\]
		and
		\[
		Z(G/K) = \langle \alpha_1 Z(G), \alpha_4 Z(G) \rangle \cong C_p \times C_p.
		\]
		Hence, by Lemma~\ref{lemma:Berkovich}, $\cd(G/K) = \{1,p\}$. Moreover, $G/K$ has nilpotency class $3$, and thus is a group of order $p^5$ belonging to either $\Phi_3$ or $\Phi_6$ (see~\cite[Subsection~4.5]{RJ}). It follows from Lemmas~\ref{lemma:phi17to20}, \ref{lemma:qotient}, and~\ref{lemma:isoGVZp^5} that $G$ admits a nonlinear irreducible character $\chi \in \nl(G)$ which is not of central type. Therefore, $G_{(17,1)}$ is not a GVZ-group, and hence no group in $\Phi_{17}$ is a GVZ-group by Theorem~\ref{thm:isoGVZ}.
		
		Next, let $G \in \Phi_{19}$. Again, we exhibit a nonlinear irreducible character that is not of central type. For instance, we have
			\begin{align*}
			G=G_{(19, 1)}=&\langle \alpha, \alpha_1, \alpha_2, \beta, \beta_1, \beta_2 : [\alpha_1, \alpha_2]=\beta, [\beta, \alpha_1]= \beta_1, [\beta, \alpha_2]= \beta_2, [\alpha, \alpha_1]=\beta_1,\\ &\alpha^p=\alpha_1^p=\alpha_2^p=\beta^p=\beta_1^p=\beta_2^p=1 \rangle,
		\end{align*}
		where $p \geq 7$. Then we have $Z(G) = \langle \beta_1, \beta_2 \rangle \cong C_p \times C_p$ and $G' = \langle \beta, \beta_1, \beta_2 \rangle \cong C_p \times C_p \times C_p$. Let $K = \langle \beta_1 \rangle \cong C_p$. Then
		\[
		(G/K)' = \langle \beta Z(G), \beta_2 Z(G) \rangle \cong C_p \times C_p
		\]
		and
		\[
		Z(G/K) = \langle \alpha Z(G), \beta_2 Z(G) \rangle \cong C_p \times C_p.
		\]
		Hence, we get $\cd(G/K) = \{1,p\}$, and $G/K$ has nilpotency class $3$. Thus, $G/K \in \Phi_3 \cup \Phi_6$. Arguing as in the previous case, we conclude that $G \in \Phi_{19}$ is not a GVZ-group.
		
		Now let $G \in \Phi_{20}$. For example, we have
		\begin{align*}
			G=G_{(20, 1)}=&\langle \alpha_1, \alpha_2, \alpha_3, \alpha_4, \alpha_5, \alpha_6 : [\alpha_5, \alpha_6]= \alpha_3, [\alpha_4, \alpha_6]=\alpha_1^{-1}, [\alpha_3, \alpha_6]=\alpha_2, [\alpha_3, \alpha_5]=\alpha_1,\\ &\alpha_1^p=\alpha_2^p=\alpha_3^p=\alpha_4^p=\alpha_5^p= \alpha_6^p=1 \rangle,
		\end{align*}
		where $p \geq 7$. Then we have $Z(G) = \langle \alpha_1, \alpha_2 \rangle \cong C_p \times C_p$ and $G' = \langle \alpha_1, \alpha_2, \alpha_3 \rangle \cong C_p \times C_p \times C_p$. Let $K = \langle \alpha_1 \rangle \cong C_p$. Then
		\[
		(G/K)' = \langle \alpha_2 Z(G), \alpha_3 Z(G) \rangle \cong C_p \times C_p
		\]
		and
		\[
		Z(G/K) = \langle \alpha_2 Z(G), \alpha_4 Z(G) \rangle \cong C_p \times C_p.
		\]
		Thus, we have $\cd(G/K) = \{1,p\}$ and $G/K$ has nilpotency class $3$. So $G/K \in \Phi_3 \cup \Phi_6$. Hence, as before, $G \in \Phi_{20}$ is also not a GVZ-group.
		
		Finally, consider $G \in \Phi_{18}$. We show that $G$ is a GVZ-group but not nested. For $p \geq 7$, we have
	\begin{align*}
		G=G_{(18, 1)}=&\langle \alpha_1, \alpha_2, \alpha_3, \alpha_4, \alpha_5, \alpha_6 : [\alpha_5, \alpha_6]= \alpha_3, [\alpha_4, \alpha_6]=\alpha_2, [\alpha_3, \alpha_6]=[\alpha_4, \alpha_5]=\alpha_1,\\ &\alpha_1^p=\alpha_2^p=\alpha_3^p=\alpha_4^p=\alpha_5^p= \alpha_6^p=1 \rangle.
	\end{align*}
		Then we have $Z(G) = \langle \alpha_1, \alpha_2 \rangle \cong C_p \times C_p$ and $G' = \langle \alpha_1, \alpha_2, \alpha_3 \rangle \cong C_p \times C_p \times C_p$.
		Let $K = \langle \alpha_1 \rangle \cong C_p$. Then we get
		\[
		(G/K)' = Z(G/K) = \langle \alpha_2 Z(G), \alpha_3 Z(G) \rangle \cong C_p \times C_p.
		\]
		Thus, $G/K$ is a group of order $p^5$ and class $2$, so $G/K \in \Phi_4$. Hence, we have $\cd(G/K) = \{1,p\}$ and
		\[
		|\nl(G/K)| = p^3 - p = |\Irr_p(G)|.
		\]
		It follows that each $\chi \in \Irr_p(G)$ is the lift of some $\bar{\chi} \in \nl(G/K)$. Since $G/K$ is a GVZ-group (see Lemma~\ref{lemma:isoGVZp^5}), each such $\chi$ is of central type by Lemma~\ref{lemma:qotient}. Therefore, $G_{(18,1)}$ is a GVZ-group, and hence so is every group in $\Phi_{18}$. Moreover, $|G/Z(G)|^{\frac{1}{2}} = p^2$, so every $\chi \in \Irr_{p^2}(G)$ is of central type with $Z(\chi) = Z(G)$ by Lemma~\ref{lemma:Special}. However, since $G/K$ is not a nested GVZ-group, there exist $\bar{\chi}_1, \bar{\chi}_2 \in \nl(G/K)$ such that $Z(\bar{\chi}_1) \neq Z(\bar{\chi}_2)$. Their lifts $\chi_1, \chi_2 \in \Irr_p(G)$ then satisfy $Z(\chi_1) \neq Z(\chi_2)$, showing that $G$ is not nested. Hence, groups in $\Phi_{18}$ are GVZ but not nested. This completes the proof of Lemma~\ref{lemma:class5}.
	\end{proof}

	\section{Proof of the main results}
	In this section, we present the proof of Theorem~\ref{thm:isoGVZp^6} and Corollary~\ref{coro:GVZcounting}, along with some immediate consequences of Theorem~\ref{thm:isoGVZp^6}.
	\begin{proof}[Proof of Theorem~\ref{thm:isoGVZp^6}]
		Let $G$ be a finite $p$-group of order at most $p^6$, where $p$ is an odd prime. By definition, every abelian group is a nested GVZ-group. In particular, the isoclinism family $\Phi_1$ consists precisely of the abelian groups of order $p^n$ for $1 \leq n \leq 6$.
		
		For groups of order $p^3$, all non-abelian groups belong to $\Phi_2$ and are nested GVZ $p$-groups. For groups of order $p^4$, there are three isoclinism families; those of nilpotency class $2$ lie in $\Phi_2$, while those of class $3$ lie in $\Phi_3$. Moreover, a group of order $p^4$ is a GVZ $p$-group if and only if it has nilpotency class $2$. In this case, it is in fact a nested GVZ $p$-group.
		
		There are $10$ isoclinism families of groups of order $p^5$, denoted $\Phi_i$ for $1 \leq i \leq 10$. A group $G$ of order $p^5$ is a GVZ-group if and only if $G \in \Phi_{1} \cup \Phi_{2} \cup \Phi_{4} \cup \Phi_{5} \cup \Phi_{7} \cup \Phi_{8}$, and $G$ is a nested GVZ-group if and only if $G \in \Phi_{1} \cup \Phi_{2} \cup \Phi_{5} \cup \Phi_{7} \cup \Phi_{8}$ (see Lemma~\ref{lemma:isoGVZp^5}).
		
		Finally, there are $43$ isoclinism families of groups of order $p^6$, denoted $\Phi_i$ for $1 \leq i \leq 43$. The result now follows by combining the results from Lemmas~\ref{lemma:class1}, \ref{lemma:class2}, \ref{lemma:class3}, \ref{lemma:class4}, and~\ref{lemma:class5}, thereby completing the proof of Theorem~\ref{thm:isoGVZp^6}.
	\end{proof}
	
	\begin{corollary}
		Let $G$ be a finite $p$-group of order at most $p^6$, where $p$ is an odd prime. Then $G$ is a VZ $p$-group if and only if $G \in \Phi_{2} \cup \Phi_{5} \cup \Phi_{15}$.
	\end{corollary}
	\begin{proof}
		The proof follows immediately from Theorem~\ref{thm:isoGVZp^6}, together with the observation that any nested GVZ-group of nilpotency class $2$ satisfying $|\cd(G)|=2$ is necessarily a VZ-group.
	\end{proof}
	
	We now proceed to the proof of Corollary~\ref{coro:GVZcounting}.
	\begin{proof}[Proof of Corollary~\ref{coro:GVZcounting}]
		A corrected classification of groups of order $p^6$ is given in~\cite{O'Brien}. Moreover, the authors determine the number of isomorphism types within each isoclinism family $\Phi_i$ for $1 \leq i \leq 43$. The result therefore follows from Theorem~\ref{thm:isoGVZp^6} together with~\cite[Table~2]{O'Brien}.
	\end{proof}
	
	Note that the classification of groups of order $p^5$ given in~\cite{RJ} for an odd prime $p$ is correct. Next, we have the following corollary, which is an immediate consequence of Theorem~\ref{thm:isoGVZp^6} together with~\cite[Subsection~4.5]{RJ}.
	\begin{corollary}
		Let $p$ be an odd prime. Then the number of isomorphism types of GVZ $p$-groups of order $p^5$ is $p+31$, and the number of isomorphism types of nested GVZ $p$-groups of order $p^5$ is $23$.
	\end{corollary}
	
	We conclude this section with Theorem~\ref{thm:existence}, which describes the range of nilpotency classes of GVZ $p$-groups of order $p^n$ and establishes the existence of a nested GVZ $p$-group of order $p^n$ for each nilpotency class within that range.
	 \begin{theorem}\label{thm:existence}
	 	Let $G$ be a GVZ $p$-group of order $p^n$, where $p$ is an odd prime. Then the nilpotency class of $G$ satisfies
	 	\[
	 	\nil(G) \leq \left\lceil \frac{n}{2} \right\rceil,
	 	\]
	 	where $\lceil \cdot \rceil$ denotes the ceiling (least integer) function. Moreover, this bound is sharp in the class of GVZ $p$-groups of order $p^n$. In particular, for each integer $c$ with
	 	\[
	 	1 \leq c \leq \left\lceil \frac{n}{2} \right\rceil,
	 	\]
	 	there exists a nested GVZ $p$-group $G$ of order $p^n$ such that $\nil(G) = c$.
	 \end{theorem}
	 \begin{proof}
	 	Let $G$ be a GVZ $p$-group of order $p^n$, where $p$ is an odd prime. Let $\chi \in \Irr(G)$. Then we have
	 	\[
	 	\chi(1) \leq p^{\left\lceil \frac{n}{2} \right\rceil - 1}
	 	\]
	 	(see Lemma~\ref{lemma:Berkovich}). Hence, observe that
	 	\[
	 	|\cd(G)| \leq \left\lceil \frac{n}{2} \right\rceil.
	 	\]
	 	Therefore, from Lemma~\ref{lemma:nilclassGVZ}, we obtain
	 	\[
	 	\nil(G) \leq \left\lceil \frac{n}{2} \right\rceil.
	 	\]
	 	
	 	Furthermore, for each positive integer $m$, there exists a nested GVZ $p$-group
	 	\begin{equation*}\label{presentation:GVZ2}
	 		G_m = \langle x, y \mid x^{p^{m+1}} = y^{p^m} = 1,\; y^{-1} x y = x^{1+p} \rangle
	 	\end{equation*}
	 	of order $p^{2m+1}$ and nilpotency class $m+1$ (see~\cite[Example~3]{LewisGVZ2}). Hence, for each integer $c$ with
	 	\[
	 	2 \leq c \leq \left\lceil \frac{n}{2} \right\rceil,
	 	\]
	 	we set $m = c - 1$ and define
	 	\[
	 	G = G_m \times C_{p^{\,n - 2m - 1}},
	 	\]
	 	where $C_{p^{\,n - 2m - 1}}$ denotes a cyclic group of order $p^{n - 2m - 1}$. By Lemma~\ref{lemma:direct product}, $G$ is a nested GVZ $p$-group of order $p^n$ and nilpotency class $c$. This completes the proof of Theorem~\ref{thm:existence}.
	 \end{proof}

	\section*{Acknowledgments}
	The author gratefully acknowledges financial support from the IISER Pune Postdoctoral Fellowship, and thanks S.~K. Prajapati and A.~Udeep for valuable discussions and insights during the preparation of the manuscript.


\end{document}